\documentclass[12pt]{amsart}
\usepackage{amsmath,amssymb,latexsym, amsthm, amscd, mathrsfs, stmaryrd} %, txfonts}
\usepackage{hyperref,mathrsfs}

\usepackage[all]{xy}

\usepackage{color}
\newcommand{\nc}{\newcommand}
\nc{\redtext}[1]{\textcolor{red}{#1}}
\nc{\bluetext}[1]{\textcolor{blue}{#1}}
\nc{\greentext}[1]{\textcolor{green}{#1}}
\nc{\red}[1]{\redtext{ #1}}
\nc{\blue}[1]{\bluetext{ #1}}
\nc{\zb}[1]{\redtext{From zb: #1}}

\newtheorem{thm}{Theorem} % [section]

\newtheorem{prop}[thm]{Proposition}
\newtheorem{conj}[thm]{Conjecture}

\theoremstyle{definition}

\newtheorem{example}[thm]{Example}

\theoremstyle{remark}
\newtheorem{rem}[thm]{Remark}

\numberwithin{equation}{section}

\newcommand{\ba}[1]{
    \begin{array}
    #1
    \end{array}
}

\newcommand{\F}{{\mathbb F}}

\newcommand{\Z}{\mathbb Z}

\newcommand{\aff}{\widehat{\mathfrak{sl}}_2}
\newcommand{\ch}{\text{char }}
\newcommand{\la}{\lambda}
\newcommand{\hf}{\frac12}

\newcommand{\V}{\text{Vir}}

\title[Conjectures on modular representations]{Some conjectures on modular representations of affine $\mathfrak{sl}_2$ and Virasoro algebra}


\begin{document}

\author[Weiqiang Wang]{Weiqiang Wang}
\address{Department of Mathematics, University of Virginia, Charlottesville, VA 22904}
\email{ww9c@virginia.edu}

\begin{abstract}
We conjecture an explicit bound on the prime characteristic of a field, under which
the Weyl modules of affine $\mathfrak{sl}_2$ and the minimal series modules of Virasoro algebra remain irreducible, and
Goddard-Kent-Olive coset construction for $\aff$ is valid. 
\end{abstract}

\maketitle

%\tableofcontents

%%%%%
%%%%%
\section{Introduction}

This note contains some speculation on modular representations of infinite-dimensional Lie algebras. 
The modular representations of infinite-dimensional Lie algebras are little understood, and in particular 
Lusztig-type conjecture (cf. \cite{Lu}) on irreducible characters in modular representation theory seems to be out of reach in the infinite-dimensional setting for now. 
We hope our explicit conjectures, though modest, might help to stimulate others to continue further in this
challenging new direction.

Let $\F$ be an algebraically closed field of characteristic $p>2$.
We speculate on an explicit bound for the prime characteristic of $\F$ such that 
the Weyl modules of the affine Lie algebra $\aff$ and the minimal series modules of Virasoro algebra
remain irreducible over $\F$.
One remarkable feature is that the two family of modules share the same bound on primes. 
This leads to another conjecture that the Goddard-Kent-Olive coset construction for $\aff$ is
valid under the same bound on $\ch\F$. 

There are two (a priori unrelated) works \cite{DR, Lai} which motivated this note. 
Lai \cite{Lai} constructed nontrivial homomorphisms between Weyl modules of $\aff$ at positive integral levels, and
showed that the modular representations change dramatically once we go beyond level one and there was no obvious conjectural bound
on the prime $p$ for which the Weyl modules remain irreducible. In \cite[Table 2]{Lai}, a list of reducible Weyl modules of $\aff$ (together with
the lowest level $\ell$ detected by the method therein relative to a given prime $p$)
is given, in which one sees the prime $p$ can increase rather quickly relative to the level $\ell$. 

Recall there are  $3$ minimal series modules of Virasoro algebra of central charge $\hf$ over ${\mathbb C}$, of highest weight $0, \hf, \frac{1}{16}$, respectively. 
In \cite{DR}, Dong and Ren showed that the minimal series modules of central charge $\hf$ remain  irreducible over $\F$,
if $\ch F \neq 2, 7$.  
The prime $7$ is bad since the values $\hf$ and $\frac{1}{16}$ coincide in $\F$ of characteristic $7$. 
%, and the representation theory of the Virasoro algebra becomes degenerate to some extent. 

Our main conjecture is that, under the assumption $\ch \F >{2\ell^2 +\ell -3}$ (for each $\ell \in \Z_{\ge 2}$), the  Weyl modules of $\aff$ of level $\ell$ and
the minimal series modules of Virasoro algebra of central charge $c_\ell$ remain irreducible, and the GKO coset construction is valid. 
This numerical bound ${2\ell^2 +\ell -3}$ ensures the highest weights of the minimal series modules of Virasoro algebra
of central charge $c_\ell$  given by \eqref{eq:cc} are distinct in $\F$ (the setting for \cite{DR} corresponds to $\ell=2$, with $c_2=\hf$ and ${2\ell^2 +\ell -3}=7$). 
We then check this bound does not contradict with constraints coming from $\aff$ in \cite{Lai}.

%%%%%
%%%%%
\section{The conjectures} 

%%%%%
\subsection{On irreducibility of Virasoro minimal series}

Let $\F$ be an algebraically closed field of characteristic $p>2$.
Recall  that the Virasoro algebra is the Lie algebra over $\F$,
$ \V =\F C \oplus \oplus_{n \in \Z}  \F L_n$,
subject to the commutation relations: $C$ is central, and
\begin{align*}
[L_m,L_n] &= (m-n)L_{m+n} + \delta_{m+n,0} \frac12 {m+1 \choose 3}C, \qquad (m,n \in \Z).
\end{align*}
Set $\V_{+} = \bigoplus_{n=1}^{\infty} \F L_n,
 \V_{-} = \bigoplus_{n=1}^{\infty } \F L_{-n}.$
Given  $c, h \in \F$, the Verma module $M_{c,h}$
over $\V$ is a free $U(\V_{-})$-module
generated by 1, such that $\V_{+} \cdot 1 = 0, L_0 \cdot 1 = h 1$ and
$ C\cdot 1 =c  1$. 
Denote by $L_{c,h}$ the unique irreducible quotient $\V$-module of $M_{c,h}$.
The scalar $c$ is called the central charge of $L_{c.h}$. 

For $\ell \in \Z_{\ge 2}$, set 
\begin{equation}
\label{eq:cc}
c_{\ell} = 1- \frac{6}{(\ell+1)(\ell+2)}.
\end{equation}
Note $c_2=\hf, c_3=\frac{7}{10}, c_4 =\frac45$.
For $1\le m\leq \ell, 1\le n\leq \ell+1$, we let
\begin{equation}
\label{eq:h}
h_{m,n} = h_{m,n; \ell} = \frac {\big(m(\ell+2) - n(\ell+1) \big)^2 - 1}{4(\ell+1)(\ell+2)}.
\end{equation}
The scalars $c_\ell$ and $h_{m,n}$ in $\F$ are understood after canceling out common integer factors
in the numerators and denominators, and 
so it is possible that these scalars are well defined even when $\ch \F$ divides $\ell+1$ or $\ell +2$. 
(For example, $c_2 =1 -\frac6{12} =\hf \in \F$ makes sense in characteristic $p=3$.)
Note that $h_{m,n} =h_{\ell+1-m, \ell+2-n}$. To avoid such a double counting, it is well known that one can simply impose
the constraint $n\le m$.
The $\V$-modules
$L_{c_\ell, h_{m,n}}\; (1\le n\le m\leq \ell%, 1\le n\leq \ell+1
)$ are usually referred to as the  (unitary) 
{\em minimal series}.

\begin{conj}
\label{conj:ms}
Let $\ell \in \Z_{\ge 2}$.   Assume $\ch \F >{2\ell^2 +\ell -3}$. 
\begin{enumerate}
\item
The minimal series $\V$-modules $L_{c_\ell, h_{m,n}} \; (\text{for } 1\le n \le m\leq \ell)$ 
over  $\F$  are irreducible. 
%\end{conj}

\item
%\begin{conj}   \label{conj:ext0}
Let $L,$ $L'$ be minimal series
$ \V$-modules over $\F$ of the same central charge $c_\ell$. 
Then  $ \mbox{Ext}^1 (L, L') = 0$. 
\end{enumerate}
\end{conj}

One can further expect that the main theorem of \cite{W} remain valid over $\F$ under the assumption that $\ch \F >{2\ell^2 +\ell -3}$;
that is, the Virasor vertex algebra $L_{c_\ell, 0}$ is rational and has the same fusion rule as in characteristic zero. 
(This is known to hold for $c_2=\hf$ by the work of Dong-Ren.)

%%%%%
\subsection{On irreducibility of Weyl modules for $\aff$}

For basics on modular representations of affine Kac-Moody algebra $\widehat{\mathfrak g}$, we refer to Mathieu \cite{Ma}; also cf. \cite{AW, Lai}. 
The level one Weyl modules of $\widehat{\mathfrak g}$ has been shown to be irreducible under the assumption that $p \geq h$ (the Coxeter number)
by various authors (deConcini-Kac-Kazhdan, Chari-Jing, Brundan-Kleshchev) 
in different ways (though all these approaches are built on the fact that the level one Weyl modules afford an explicit combinatorial/vertex operator
realization); cf. \cite{Lai} for references.

Denote by $X^+_\ell$ the set of dominant integral weights of level $\ell \in \Z_{\ge 1}$. 
Note that $X^+_1 =\{\omega_0, \omega_1\}$ consists of 2 fundamental weights. 
Then $X^+_\ell =\{\la_{\ell;n} = (\ell -n) \omega_0 +n \omega_1 \mid 0 \le n \le \ell \}$.  
%We shall denote $\la_{n;\ell} =(\ell -n) \omega_0 +n \omega_1$, for $0\le n\le \ell$. 
We can define the Weyl module $V(\la)$, for $\la \in X^+_\ell$ (for $\ell \in \Z_{>0}$) of the
affine Lie algebra $\aff$ over $\F$ as in \cite{Ma}. 

\begin{conj}
\label{conj:Weyl}
Let $\ell \in \Z_{>0}$. Assume $\ch \F >{2\ell^2 +\ell -3}$. %(\ell+1)^2$. 
\begin{enumerate}
\item
The Weyl modules $V(\la)$ of $\aff$ over $\F$   are irreducible, for $\la \in X^+_\ell$. 

\item
Let $\la, \mu \in X^+_\ell$. 
Then  $ \mbox{Ext}^1 (V(\la), V(\mu)) = 0$. 
\end{enumerate}
\end{conj}
One can rephrase Conjecture~\ref{conj:Weyl} as that the  category of rational representations of level $\ell$ of the Kac-Moody group (associated to $\aff$)
over $\F$ is semisimple if $\ch \F >{2\ell^2 +\ell -3}$.

One can further hope that  under the assumption that $\ch \F >{2\ell^2 +\ell -3}$ the
affine vertex algebra $V(\ell \omega_0)$ is rational and has the same fusion rule as over the complex field ${\mathbb C}$  (which was  
computed by I.~Frenkel and Y.~Zhu).

%%%%%
\subsection{Modular GKO coset construction}

Let $\ell \ge 2$.
Recall the Goddard-Kent-Olive $(\aff |_{\ell -1} \oplus \aff |_1, \aff |_{\ell})$-coset construction  \underline{over ${\mathbb C}$} \cite{GKO}  refers to the following
tensor product decomposition into a direct sum of multiplicity-free irreducible $(\aff |_\ell, \V)$-modules: 
\begin{align}
\begin{split}
\label{eq:GKO}
& V(\la_{\ell-1; n}) \otimes V(\omega_\epsilon)
\\
&\quad = \bigoplus_{\stackrel{0\le j \le n}{j \equiv n+\epsilon \; (\text{mod } 2)}} V(\la_{\ell; j}) \otimes L_{c_\ell, h_{n+1,j+1}}
\bigoplus \bigoplus_{\stackrel{n+1\le j \le \ell}{j \equiv n+\epsilon \; (\text{mod } 2)}} V(\la_{\ell; j}) \otimes L_{c_\ell, h_{\ell-n,\ell+1-j}}
\end{split}
\end{align}
for all $%\la_{n; \ell-1} \in X^+_{\ell-1}
0\le n\le \ell-1 \text{ and }  \epsilon\in \{0,1\}.$

\begin{conj} [Modular GKO  conjecture]
\label{conj:GKO}
%The GKO $(\aff |_{\ell -1} \oplus \aff |_1, \aff |_{\ell})$-coset construction is valid over $\F$, that is, 
The multiplicity-free  decomposition \eqref{eq:GKO} into a direct sum of irreducible $(\aff |_\ell, \V)$-module is valid over $\F$,
 if $\ch \F > {2\ell^2 +\ell -3}$. %(\ell+1)^2$. 
\end{conj}

\begin{rem}
The following ``partial semisimple tensor product" statement is a 
consequence of  Conjectures~\ref{conj:Weyl} and ~\ref{conj:GKO}:
{\em 
Assume $\ch \F >{2\ell^2 +\ell -3}$.  For any positive integers $\ell_1, \ell_2$ such that $\ell_1 +\ell_2 \le \ell$
and any $\la \in X^+_{\ell_1}, \mu \in X^+_{\ell_2}$, $V(\la) \otimes V(\mu)$ is a semisimple $\aff$-module. 
}
\end{rem}

%%%%%
%%%%%
\section{Evidence and generalizations}

%%%%
\subsection{Supporting evidence for Conjecture~\ref{conj:ms}}

%Let $\ell \in \Z_{\ge 2}$. 
Recall $h_{m,n}$ from \eqref{eq:h}.
A prime $p$ is called  %{\bf $\bf \VV |_{c_\ell}$-good} (or for short, 
{\bf $\bf \V |_{c_\ell}$-good}  if the scalars $h_{m,n}$
$(1\le n\le m\leq \ell)$ are pairwise distinct  %except the usual identification $h_{m,n} =h_{\ell+1-m, \ell+2-n}$,
(and hence there are $\ell(\ell+1)/2$ distinct values of such $h_{m,n}$); otherwise, a prime $p$ is called {\bf $\bf \V |_{c_\ell}$-bad}.

\begin{prop} 
\label{prop:h}
Every prime greater than  %$(\ell+1)^2$ 
${2\ell^2 +\ell -3}$ are $\V |_{c_\ell}$-good.
Equivalently, every $\V |_{c_\ell}$-bad prime does not exceed ${2\ell^2 +\ell -3}$. 
\end{prop}
Note $2\ell^2 +\ell -3 =(2\ell+3)(\ell-1)$ is a prime  only when $\ell=2$. 

\begin{proof}
Let $1\le m, m' \leq \ell, 1\le n, n' \leq \ell+1$. 
Denote the numerator of $h_{m,n}$ by 
$$\tilde h_{m,n}=  \big(m(\ell+2) -n(\ell+1) \big)^2 - 1,
$$
and denote 
 \begin{equation}
 \label{Dpm}
 D_{m,n;m',n'}^{\ell, \pm} = (m \pm m')(\ell+2) - (n\pm n')(\ell+1).
 \end{equation}
Then
$$\tilde h_{m,n}  - \tilde h_{m',n'} = D_{m,n;m',n'}^{\ell, +} D_{m,n;m',n'}^{\ell, -}.
$$
Hence $\tilde h_{m,n}  = \tilde h_{m',n'}$ if and only if 
\Big((i) $D_{m,n;m',n'}^{\ell, +} =0 $ or (ii) $D_{m,n;m',n'}^{\ell, -} =0$\Big).
Let us consider (i) and (ii) as equations over $\Z$ for now. 
Assume  (i) holds.
Since $\{\ell+1, \ell+2\}$ are relatively prime, we have $(\ell+2) \mid (n+n')$ and $(\ell+1) \mid (m+m')$,  which further imply
that $n+n' =\ell+2$ and $m+m' =\ell+1$, respectively (recall $m,m' \le \ell$,  and $n, n' \le \ell+1$). 
Similarly, (ii) holds imply $m=m'$ and $n=n'$. 

Denote
\begin{equation}
 \label{eq:setD}
B_\ell =\Big\{\big |D_{m,n;m',n'}^{\ell, +}\big| \; \Big\vert  \; 1\le m, m' \leq \ell, 1\le n, n' \leq \ell+1 \Big \}.
\end{equation}
(A set defined using $|D_{m,n;m',n'}^{\ell, -}|$ instead is equivalent to $B_\ell$
thanks to $\tilde h_{m,n} =\tilde h_{\ell+1-m, \ell+2-n}$.)
Note the maximal value in $B_\ell$ is achieved at $(\ell+\ell) (\ell+2) - (1+1)(\ell+1) =2(\ell^2 +\ell-1)$ which is manifestly even;
the second largest absolute value is achieved at 
$(\ell+\ell) (\ell+2) - (1+2)(\ell+1) =2\ell^2 +\ell -3$; so
under the assumption $\ch\F>2\ell^2 +\ell -3$, all values of $\tilde h_{m,n}$ for $1\le n \le m \le \ell$ are distinct. 
\end{proof}

%\begin{enumerate} \item
Conjecture~\ref{conj:ms} is known to hold when $c_2=\frac12$ (i.e., $\ell=2$) \cite{DR}. 
A basic observation of \cite{DR} can be rephrased  that
the $\V |_{c_2}$-bad primes are $\{2, 7\}$.
Note that the $\V |_{c_2}$-good primes $p=3, 5$  are not detected by Proposition~\ref{prop:h}; see Remark~\ref{rem:except} below for an explanation of these missing primes.

%%%%
\subsection{Supporting evidence for Conjecture~\ref{conj:Weyl}}

Note \cite[Table 2]{Lai} provides a list of reducible Weyl modules (detected by the approach therein) of lowest levels $\ell$ at a given prime $p$.
Table~1 below is a somewhat novel look at the data provided in \cite[Table 2]{Lai}, and it indicates the maximal known prime $p$ for which
there exists a reducible Weyl module at level $\ell \ge 2$. 

\begin{table}[!h]
\caption{Maximal known primes $p$ for reducible Weyl modules at  level $\ell$}
\[
\ba{{|c|ccccccc|}
\hline
\ell &2 & 3 &4 &5 &6  &7  &8 \\
\hline
p   &3 & 13&11&23 &37  & 47  & 53 \\
%\hline
%(\ell+1)^2   &9& 16&25&36 &49  & 64  & 81 \\
\hline
2\ell^2 +\ell-3   &7& 18&33&52 &75  & 102  & 133 \\
\hline
}
\]
\end{table}
We note that $p<2\ell^2 +\ell-3$ for all $\ell$ in the table, and so Conjecture~\ref{conj:Weyl} is consistent with the
results of \cite{Lai}. 

\begin{rem}
Lai and the author have formulated a (conjectural) linkage principle; see \cite[Conjecture~6.1]{Lai}. If one can show that the weights in $X^+_\ell$ are minimal 
in the Bruhat order in each linkage class assuming $\ch \F >{2\ell^2 +\ell -3}$, then Conjecture~\ref{conj:Weyl} would follow (modulo
the linkage principle conjecture).
\end{rem}

%%%%
\subsection{Supporting evidence for Conjecture~\ref{conj:GKO}}

The evidence from Table~1 (based on \cite[Table~2]{Lai})  for affine algebra $\aff$
is remarkably consistent with the constraints from Proposition~\ref{prop:h} for Virasoro algebra.
Conjecture~\ref{conj:GKO} offers a reasonable and conceptual way of explaining such a coincidence, and
it helps to relate  Conjecture~\ref{conj:Weyl}  and Conjecture~\ref{conj:ms}.
 
%%%%
\subsection{More precise bound on $\ch \F$ for Virasoro algebra}

One could make a conjecture (which strengthens Conjecture~\ref{conj:ms}) that whenever $\ch\F$ is a $\V|_{c_\ell}$-good prime
the minimal series $L_{c_\ell, h_{m,n}}$  are irreducible. 
One drawback of this stronger conjectural bound of $\ch\F$ 
is that a precise description of the set of $\V|_{c_\ell}$-bad primes is difficult for general $\ell \ge 2$,
in contrast to the explicit though coarser bound in Proposition~\ref{prop:h}. 
We will provide some partial answer below. The $\V|_{c_\ell}$-bad primes for small values of $\ell$ can be computed by hand. 

\begin{example}
The $\V|_{c_2}$-bad primes are  $\{2, 7\}$ (i.e., the primes $\le 2\ell^2+\ell -3=7$ except $3, 5$).
 
The $\V|_{c_3}$-bad primes are  $\{2, 3,  7, 9,  13, 17\}$ (i.e., primes $\le 2\ell^2+\ell -3=18$ except $5, 11$).
 
 The $\V|_{c_4}$-bad primes are  %$\{2, 3, 7, 9, 11, 13, 17,   23\}$ (i.e., 
\{all primes $\le 2\ell^2+\ell -3=33\} \backslash \{5, 19, 29, 31\}$.

The $\V|_{c_5}$-bad primes are %all primes $\leq 37$  except $7, 29$ i.e., 
\{all primes $\le 2\ell^2+\ell -3=52\} \backslash \{7, 29, 41, 43, 47\}$.

The $\V|_{c_6}$-bad primes are \{all primes  $\le 2\ell^2+\ell -3=75 \} \backslash \{7, 41, 71, 73\}$.
\end{example}

For integers $a, b$ with $a\le b$, denote by $[a,b]$ the interval of integers between $a$ and $b$.
%consisting of integers $k$ with $a\le k \le b$. 
 Set 
 $$B_\ell(a) =[\ell^2+\ell +a(\ell+2), \ell^2+ 2\ell -1 +a(\ell+1) ], \quad \text{ for } 0\le a \le \ell-1.
 $$
Note $k <k'$ for all $k \in B_\ell(a), k' \in B_\ell(a')$ whenever $a<a'$. 
%Recall $B_\ell$ from \eqref{eq:setD}. 

 \begin{prop}
 \label{prop:X}
The set $B_\ell$ \eqref{eq:setD} is given by
$$B_\ell =[1, \ell^2+\ell-2] \bigcup B_\ell(0) \bigcup B_\ell(1) \bigcup \cdots \bigcup  B_\ell(\ell-1).
$$ 
\end{prop}

\begin{proof}
Recall $B_\ell$ is defined using $|D_{m,n;m',n'}^{\ell, +}|$; cf. \eqref{Dpm}.
We first list the values of $(n+n')(\ell+1)$, for $1\le n,n' \le \ell+1$, in an increasing row  (there are $(2\ell+1)$ entries), 
and list the values of $(m+m')(\ell+2)$, for $1\le m,m' \le \ell$, in an increasing column (there are $(2\ell-1)$ entries).
By taking the absolute value of the difference of row and column entries, 
one produces a $(2\ell-1) \times (2\ell+1)$ matrix $A$ whose entries are given by $|D_{m,n;m',n'}^{\ell, +}| = |(m + m')(\ell+2) - (n+ n')(\ell+1)|.$
One observes that the matrix is symmetric under rotation by 180 degrees so we only need to consider the $(2\ell-1) \times (\ell+1)$ submatrix, denoted by $D$,
which consists of the $(\ell+1)$ columns of $A$. 
By listing the entries at the $r$th diagonals of $D$ in the following order:  $r=1, 0, 2, -1, 3, -2, \ldots, 2-\ell, \ell$, one obtains
exactly the interval $[1, \ell^2+\ell-2]$. (This is not so miraculous by noting the following: the values in each diagonal form a natural sequence, 
 the last column of $D$ is symmetric by flipping.) Now we are left with the lower $\ell$ diagonals of $D$, whose values are given by
 the $\ell$ intervals $B_\ell(a)$, for $0\le a \le \ell-1$.

%\begin{example}
The easiest way for a reader to convince herself/himself of the above proof is to work out an example for one particular $\ell$. For example let $\ell=5$.
Following the recipe in the proof above, we obtain  the initial row and column vectors to be
$(12, 18, 24, 30, 36, 42, \ldots, 72)$ and $(14, 21, 28, 35, 42, 49, 56,63, 70)^t$. This leads to the following matrix
% (where the numbers are marked bold alternatively to make diagonals more recognizable):
$$
D=\begin{bmatrix}
   2 &{\bf 4}    &10   &{\bf 16}    &22   &{\bf 28} \\
   {\bf 9} &3    &{\bf 3}     &9      &{\bf 15}    &21  \\
16  &{\bf 10}  &4     &{\bf 2}     &8      &{\bf 14}  \\
 {\bf 23} &17  &{\bf 11}   &5     &{\bf 1}      &7  \\
 30 &{\bf 24}  &18   &{\bf 12}   &6     &{\bf 0}  \\
 {\bf 37} &31  & {\bf 25}   &19   & {\bf 13}   &7  \\
 44 & {\bf 38}  &32   &  {\bf 26}  &20   & {\bf 14}  \\
  {\bf 51} &45  & {\bf 39}   &33   & {\bf 27}   &21  \\
58  &{\bf 52}  &46   & {\bf 40}   &34   & {\bf 28}
\end{bmatrix}.
$$
Then one sees clearly that the above recipe leads to the statement in the proposition. 
%\end{example}
\end{proof}

Note that $B_\ell (\ell-1) =\{2(\ell^2 +\ell -1)\}$ consists of  the largest integer in  $B_\ell$; moreover
$2\ell^2+\ell-3 \in B_{\ell}(\ell-2)$ is the  largest integer in  $B_\ell' :=B_\ell \backslash B_{\ell}(\ell-1)$, or the second largest integer in  $B_\ell$. 
By definition of $B_\ell$ and Proposition~\ref{prop:h},
we have 
\begin{equation}
\label{primeX}
\{\V |_{c_\ell}\text{-bad primes} \} \subseteq [1,  {2\ell^2 +\ell -3}] \cap B_\ell.
\end{equation}
One may regard \eqref{primeX} as a sharper form of the description of the bound in Proposition~\ref{prop:h}, thanks
to the concrete description of the set $B_\ell$ in  Proposition~\ref{prop:X}.

\begin{rem}
\label{rem:except}
Note $(\ell+1)^2 \not \in B_\ell$ and $(\ell+2)^2 \not \in B_\ell$.
Recall the denominator for $h_{m,n}$ is $4(\ell+1)(\ell+2)$. 
It follows that if either $\ell+1$ or $\ell+2$ happens to be a prime it must be a $\V |_{c_\ell}$-good prime.
In the above examples for small $\ell$, this prime happens to be the smallest $\V |_{c_\ell}$-good prime,
 and the second $\V |_{c_\ell}$-good prime happens to be $\ell^2+\ell-1$ (where $\ell^2+\ell-1$ happens to be a prime).
 \end{rem}
 
 Introduce $G_\ell = [1, 2\ell^2+\ell-3] \backslash B_\ell$.
If follows by definition that the set of $\V |_{c_\ell}$-good primes $\leq 2\ell^2+\ell-3$ is contained in $G_\ell$ 
(recall all primes $>2\ell^2+ \ell-3$ are $\V |_{c_\ell}$-good by Proposition~\ref{prop:h}).
Set 
$$G_\ell(a) =[\ell^2+\ell -1+a(\ell+1), \ell^2+ \ell -1 +a(\ell+2) ], \quad \text{ for } 0\le a \le \ell-1.
$$
Note $k <k'$ for all $k \in G_\ell(a), k' \in G_\ell (a')$ whenever $a<a'$. 
Then one derives by definition and Proposition~\ref{prop:X} that
$$G_\ell = G_\ell(0) \bigcup G_\ell (1) \bigcup \cdots \bigcup  G_\ell (\ell-1).
$$

%%%%
\subsection{Discussions} 
%Higher ranks}
We indicate below several possible generalizations, provided the conjectures of this paper are valid.

One can ask the same question in Conjecture~\ref{conj:ms}  for non-unitary minimal series of Virasoro algebra. 
An analysis analogous to Proposition~\ref{prop:h} should allow one to find a conjectural bound on the characteristic of the field $\F$ 
over which the non-unitary minimal series of Virasoro algebra remain irreducible. 

It would be interesting to generalize the conjectures of this paper to affine algebras of higher ranks and W-algebras 
(in particular for $\widehat{\mathfrak{sl}}_n$ and the corresponding $W$-algebra $W_n$).
One could  try to 
find a conjectural bound on $\ch\F$ under which the Weyl modules are irreducible,   by analyzing the highest weights of the unitary minimal series 
of the W-algebras \cite{FKW}. 
One can see by a similar method as in Proposition~\ref{prop:h} this bound should be some quadratic polynomial on the level and the rank.
One can also derive some more evidence from \cite[Theorem~5.10]{Lai} on
$\ch\F$ for reducible Weyl modules of affine Lie algebras of higher rank (say, $\widehat{\mathfrak{sl}}_3$). 
Then one would check if the bound arising from W-algebra minimal series is compatible with the bound from
affine Lie algebras.

%%%%\subsection{Other generalizations}

Instead of Virasoro algebra, one can consider the super-Virasoro algebra, also known as Neveu-Schwarz (and Ramond) algebras, and
its unitary minimal series. A similar analysis can lead to a conjectural bound on $\ch\F$ under which 
the minimal series of the super-Virasoro algebra are irreducible. 
In the same way that affine Lie algebra $\aff$ is related to Virasoro algebra via the GKO construction,
the Neveu-Schwarz algebra is related to the affine Lie superalgebra $\widehat{\mathfrak{osp}}_{1|2}$.
So we can give a conjectural bound on $\ch\F$ (relative to the levels) under which   the  Weyl modules of 
$\widehat{\mathfrak{osp}}_{1|2}$ are irreducible. 

 \vspace{.3cm}
 
 {\bf Acknowledgement.} I thank Chongying Dong for stimulating discussion on \cite{DR} and Chun-Ju Lai
 from whom I learned about \cite{Lai}  firsthand; this note is influenced by both works.  
The author  is partially supported by the NSF grant DMS-1405131.


\end{document}